\documentclass[12pt, final]{amsart}
\usepackage{geometry}
\usepackage{amssymb}
\usepackage{pgf,tikz}
\usepackage{mathrsfs}
\usetikzlibrary{arrows}
\usepackage{stmaryrd}
\usepackage{graphicx}
\usepackage[colorlinks=true, citecolor=blue, linkcolor=blue]{hyperref}
\usepackage{ifdraft}
\ifdraft{
\newgeometry{a4paper, asymmetric, top=2.5cm, bottom=2.5cm, left=2cm, right=3cm, 
    marginparwidth=3cm, marginparsep=3pt}
\usepackage[bordercolor=white, color=white, colorinlistoftodos]{todonotes}
\usepackage[notref, notcite]{showkeys}

}{
\usepackage[disable]{todonotes}
}
\newcommand{\HOX}[1]{\todo[noline, size=\tiny]{#1}}

\makeatletter\providecommand\@dotsep{5}\def\listtodoname{List of Todos}\def\listoftodos{\hypersetup{linkcolor=black}\@starttoc{tdo}\listtodoname\hypersetup{linkcolor=blue}}\makeatother

\newtheorem{lemma}{Lemma}

\newtheorem{theorem}{Theorem}

\newtheorem{remark}{Remark}

\def\R{\mathbb R}

\def\N{\mathbb N}

\def\H{\mathcal H}

\def\p{\partial}
\DeclareMathOperator{\supp}{supp}
\DeclareMathOperator{\Div}{div}

\newcommand{\pair}[1]{\left\langle #1 \right\rangle}
\newcommand{\norm}[1]{\left\|#1 \right\|}

\def\inter{\text{int}}

\DeclareMathOperator{\WF}{WF}

\def\d{\downarrow}
\def\u{\uparrow}

\def\H{\mathcal H}

\ifdraft{
\date{May 2015, last edit by Lauri, compiled \today}
}

\title[Time reversal with stabilizing boundary conditions for PAT]{Time reversal method with stabilizing boundary conditions for Photoacoustic tomography}
\author{Olga Chervova}
\address{Department of Mathematics, University College London, Gower Street, London UK, WC1E 6BT.}
\email{o.chervova@ucl.ac.uk}
\author{Lauri Oksanen}
\address{Department of Mathematics, University College London, Gower Street, London UK, WC1E 6BT.}
\email{l.oksanen@ucl.ac.uk}

\begin{document}
\begin{abstract}
We study an inverse initial source problem 
that models Photoacoustic tomography measurements with array detectors,
and introduce a method that can be viewed as a modification of the so called back and forth nudging method. We
show that the method converges at an exponential rate under a natural visibility condition, with data given only on a part of the boundary of the domain of wave propagation. In this paper we consider the case of noiseless measurements. 
\end{abstract}
\maketitle

\section{Introduction}

We consider a mathematical model for the emerging hybrid imaging method called the Photoacoustic tomography (PAT). 
Hybrid imaging methods are based on a physical coupling 
between two types of waves, and 
in the case of PAT, the coupling is the photoacoustic effect,  that is, the conversion of electromagnetic energy, absorbed by a specimen, into the energy of propagating acoustic waves.
The electromagnetic radiation is produced (typically by a laser), and the acoustic pressure is recorded, outside the specimen. 

The time scales of the electromagnetic energy absorption 
and the acoustic wave propagation are different enough, so that the acoustic propagation can be modelled as an initial value problem for the wave equation, the initial source being proportional to the absorbed electromagnetic energy. 
The rationale behind PAT is that it combines 
the high contrast in
electromagnetic absorption, say between healthy and cancerous tissue in some specific energy bands, with the
high resolution of ultrasound. 
For more information on PAT, see \cite{Wang2009, PAT_review, Xu2006} and the mathematical reviews \cite{Kuchment2008,scherzer2010handbook}.

In the present work we consider a model of PAT measurements using array detectors. The difference from the more traditional point detector case is that 
the detectors act as reflectors for the acoustic waves  \cite{ellwood2014use, huang2013improving}. For mathematical studies on PAT with the reflecting walls see \cite{Cox2007, holman2015gradual, kunyansky2013photoacoustic}.
A typical array detector is a rectangle, and this motivates us to consider measurement of acoustic pressure on $(0,T) \times \Gamma_0$, where 
$T > 0$ and 
\HOX{Ask a ref from Ben}
\begin{align}
\label{typical_Gamma}
\Gamma_0 = \{(x^1, x^2, x^3) \in \p \Omega_0;\ x^j \le 1 \}
\end{align}
models three square arrays located in the corner of the octant
$$
\Omega_0 = \{(x^1, x^2, x^3) \in \R^3;\ x^j > 0 \},
$$
see Figure \ref{Figure 3dObject}. Recently a similar geometry, namely, two orthogonal array detectors forming a V-shape, was studied experimentally in \cite{ellwood2015orthogonal}.

\begin{figure}
  \includegraphics[width=0.5\linewidth]{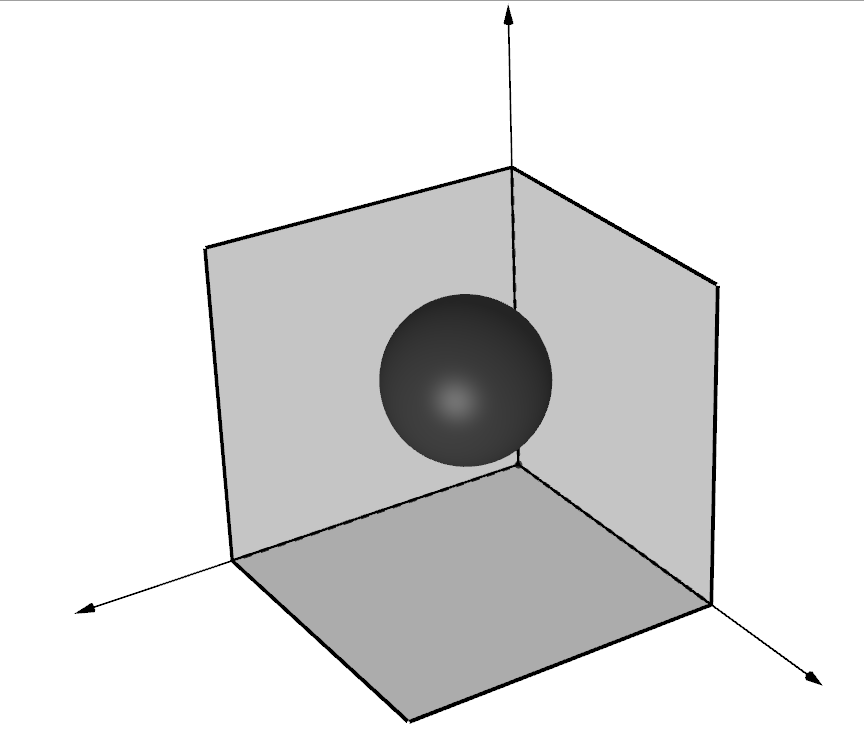}
  \caption{Three square arrays, modelled by $\Gamma_0$, in light grey. The dark grey ball depicts the specimen to be imaged.}
  \label{Figure 3dObject}
\end{figure}

In order to avoid mathematical technicalities,
we will consider wave propagation in an unbounded 
open and connected set $\Omega \subset \R^n$ with smooth boundary. Such $\Omega$ can be seen as an approximation of $\Omega_0$.
We assume that a strictly positive $c \in C^\infty(\overline \Omega)$ gives the speed of sound in $\Omega$.
A PAT measurement 
on an open set $\Gamma \subset \p \Omega$ 
for $T > 0$ time units
is then modelled by 
the trace $w|_{(0,T) \times \Gamma}$
where $w$ solves the wave equation 
\begin{align}
\label{eq_wave_PAT}
\begin{cases}
\p_t^2 w - c^2 \Delta w = 0, &\text{in $(0,T) \times \Omega$},
\\
\p_\nu w  = 0,
 &\text{in $(0,T)  \times \p \Omega$},
\\
w|_{t=0} = w_0,\ \p_t w|_{t=0} = 0,
&\text{in $\Omega$},
\end{cases}
\end{align}
and $\supp(w_0) \subset \Omega$ is assumed to be compact.
The objective is to:
\begin{itemize}
\item[(P)] Recover $w_0$ given $w|_{(0,T) \times \Gamma}$.
\end{itemize}

The finite speed of propagation property satisfied by the solution of (\ref{eq_wave_PAT})
implies that if $\supp(w_0)$ is large in comparison with $\Gamma$ and $T$,
then $w|_{(0,T) \times \Gamma}$ does not determine $w_0$ uniquely. 
Indeed, the initial pressure $w_0$ is determined by the measurement $w|_{(0,T) \times \Gamma}$ if and only if 
$d(x, \Gamma) \le T$ for all $x \in \supp(w_0)$. Here $d(\cdot, \cdot)$ is the Euclidean distance 
if $c=1$ identically and $\Omega$ is convex and, in general, it is the distance function on the Riemannian manifold with boundary $(\Omega, c^{-2} dx^2)$.

In this paper, we introduce a time reversal method that recovers 
$w_0$ 
given the measurement 
$w|_{(0,T) \times \Gamma}$,
assuming the sharp condition $d(\supp(w_0), \Gamma) \le T$.
Moreover, under the further assumption that all the singularities of $w_0$ propagate to $\Gamma$ in time $T$,
see the condition (VC) below for the precise formulation,
we show that the method converges at an exponential rate.
To our knowledge, this is the first method with proven exponential convergence rate in the case
that measurements are given on a bounded set and the domain of wave propagation is unbounded.

\subsection{Previous literature}

Our time reversal method is a part of the tradition \cite{Haine2014, Ito2011, Ramdani2010} of 
back and forth nudging methods originating from \cite{Auroux2005}.
It is also similar to the tradition of time reversal  methods \cite{Finch2004, Hristova2009, Hristova2008} and \emph{iterative} time reversal methods 
\cite{Acosta2015, Qian2011, Stefanov2009a, Stefanov2011a, Stefanov2015b}.

The closest work to the present one is 
\cite{Nguyen2015} where a time reversal method 
similar to ours is introduced in the context of PAT.
The method in \cite{Nguyen2015} uses the same stabilizing boundary conditions that we are using, see (\ref{eq_wave_nudge}) below, and although the result \cite{Nguyen2015} was obtained independently from the tradition originating from \cite{Auroux2005}, it can be viewed as a part of this tradition in the sense that it fits in the abstract setting \cite{Ramdani2010}. 

The result \cite{Nguyen2015} assumes the  geometric control condition by Bardos, Lebeau and Rauch \cite{Bardos1992},
and this is a global assumption on the whole domain  $\Omega$. 
The main novelty in the present work is that our geometric condition, see (VC) below, is local in the following sense: suppose that the accessible part of the boundary $\Gamma$ is strictly convex (in the sense of the second fundamental form), then our method converges exponentially if $w_0$ is supported in a sufficiently small neighbourhood of $\Gamma$.
In the case that $c=1$ identically, this neighbourhood is the convex hull of $\Gamma$. Moreover, our result also applies to cases where $\Gamma$ is convex but not strictly convex, an example being (a smooth approximation of) $\Gamma_0$ in Figure \ref{Figure 3dObject}.

We prove the exponential convergence by using microlocal analysis, see Lemma \ref{lem_microlocal} below, and this step can be viewed as a local analogue of \cite{Bardos1992} under the local condition (VC).
Let us also point out that the condition (VC) does not use signals reflected from the inaccessible part of the boundary $\p \Omega \setminus \Gamma$ and, in particular, it does not require the domain $\Omega$ to be a closed cavity. This is motivated by practical applications 
since the signal from $w_0$ typically deteriorates badly after reflections from $\p \Omega \setminus \Gamma$.
For example, in recent experimental study \cite{ellwood2014use} virtually nothing could be detected after the first two reflections even in media with low acoustic absorption (water in the case of \cite{ellwood2014use}). See also \cite{Cox2007, cox2009photoacoustic} for other experimental studies in case when $\Omega$ is a closed cavity.

Of the above mentioned methods, only \cite{Haine2014} is applicable  to the case that $\Omega$ is unbounded. 
The main difference between \cite{Haine2014} and the present work is that we show that our method converges to $w_0$, whereas in \cite{Haine2014} convergence is shown only up to an abstract projection that is not characterised in terms of Sobolev spaces.

Let us also point out that for the iterative time reversal methods that do not use stabilizing boundary conditions
\cite{Stefanov2009a,Stefanov2011a,Stefanov2015b},
exponential convergence has been shown only in the full data case, that is, when the support of $w_0$ is enclosed by $\Gamma$. Of these three methods, \cite{Stefanov2015b} is the closest to the present one. There the problem (P) is considered in a bounded domain $\Omega$ with Neumann boundary conditions, and the method is shown to converge exponentially when $\Gamma = \p \Omega$.

\section{Statement of the results}

Let us formulate our time reversal method 
in a slightly more general context than (\ref{eq_wave_PAT}).
Let $(M, g)$ be a smooth Riemannian manifold with boundary,
and let $\mu \in C^\infty(M)$ be strictly positive.
Consider the wave equation,
\begin{align}
\label{eq_wave_meas}
\begin{cases}
\p_t^2 v - \Delta_{g,\mu} v = 0, &\text{in $(-T,T) \times M$},
\\
\p_\nu v  = 0,
 &\text{in $(-T,T)  \times \p M$},
\\
 v|_{t = 0} = v_0,\ \p_t v|_{t = 0} = v_1, &\text{in $M$},
\end{cases}
\end{align}
where $\Delta_{g,\mu}$ and $\p_\nu$ are the weighted Laplace-Beltrami operator
and the normal derivative,
\begin{align*}
\Delta_{g,\mu}u = \mu^{-1} \Div(\mu \nabla u),
\quad \p_\nu u = \left\langle \nu, \nabla u\right\rangle.
\end{align*}
Here $\nu$ is the outward pointing unit normal vector field on $\p M$, $\Div$ and $\nabla$ are the
divergence and gradient on $(M,g)$, and $\langle \cdot, \cdot
\rangle$ denotes the inner product with respect to the metric
$g$. Occasionally, we use also the notation $|\xi|^2 = \pair{\xi, \xi}$,
$\xi \in T_x M$, $x \in M$.

Note that the equation (\ref{eq_wave_PAT}) is a special case of (\ref{eq_wave_meas}).
Indeed, as $\p_t w|_{t=0} = 0$, the function 
$$
v(t,x) = 
\begin{cases}
w(t,x), & t \in (0,T),
\\
w(-t,x), & t \in (-T,0),
\end{cases}
$$
satisfies (\ref{eq_wave_meas}), 
with 
$g = c^{-2}dx^2$, $\mu = c^{n-2}$,
$v_0 = w_0$ and $v_1 = 0$,
when $w$ satisfies (\ref{eq_wave_PAT}).
Let $\Gamma \subset \p M$ be open and define 
$$
\Lambda V_0 = v|_{(-T,T) \times \Gamma}, \quad V_0 = (v_0, v_1),
$$
where $v$ is the solution of (\ref{eq_wave_meas}).
The PAT problem (P) is a special case of the inverse initial source problem
to recover $V_0$ given $\Lambda V_0$.

We will next introduce the function spaces and operators used in our time reversal iteration. 
We write for a compact set $\mathcal K \subset M$,
$$
\H(\mathcal K) = \{(u_0, u_1) \in H^1(M) \times L^2(M);\ \supp(u_j) \subset \mathcal K,\ j = 0,1 \}.
$$
When $\mathcal K \ne M$, we equip $\H(\mathcal K)$ with the energy norm defined by
\begin{align}
\label{energy_norm}
\norm{U}_*^2 
= \frac 1 2 \int_M ( |\nabla u_0|^2 + |u_1|^2 )\, \mu dx,
\quad U = (u_0, u_1),
\end{align}
where $dx$ stands for the Riemannian volume measure on $(M,g)$.
By \cite{Lasiecka1989} the map $\Lambda$ is continuous from $\H(\mathcal K)$ to $H^1((-T,T) \times \Gamma)$ when $\mathcal K \subset M^\inter$.

Note that $\norm{\cdot}_*$ fails to be norm on $\H(\mathcal K)$ if $\mathcal K = M$, since there are constant non-zero functions $U \in \H(M)$ and $\norm{U}_* = 0$ for them. 
To avoid technicalities related to this, we will make the standing assumption that 
$M$ is unbounded.
This assumption reflects also the fact that we are mainly interested in measurement geometries similar to (\ref{typical_Gamma}).

Let $\Lambda V_0$ be a fixed measurement 
where $V_0 \in \H(K)$ for a compact set $K \subset M^\inter$.
Define the time intervals $I^+ = (0,T)$ and $I^- = (-T,0)$, and write 
$$
t_\u^+ = 0, \quad t_\d^+ = T,
\quad
t_\u^- = -T, \quad t_\d^- = 0,
$$
for the boundary points of the intervals $I^\pm$.
We choose a non-negative cut off function $\chi_0 \in C_0^\infty(\p M)$ satisfying $\supp(\chi_0) = \overline \Gamma$, and consider the wave equation 
\begin{align}
\label{eq_wave_nudge}
\begin{cases}
\p_t^2 u - \Delta_{g,\mu} u = 0, &\text{in $I^\pm \times M$},
\\
\p_\nu u  + \chi\p_t u = \chi \partial_t\Lambda V_0,
 &\text{in $I^\pm  \times \p M$},
 \\
 u|_{t = t_{\u\d}^\pm} = u_0,\ \p_t u|_{t = t_{\u\d}^\pm} = u_1, &\text{in $M$},
\end{cases}
\end{align}
where $\chi = \chi_0$ in the two cases $t_\u^\pm$
and $\chi = -\chi_0$ in the two cases $t_\d^\pm$.
Note that (\ref{eq_wave_nudge}) is solved forward in time in the former two cases, and backward in time in the latter two.
We write  
$$
T_\u^+ = T, \quad T_\d^+ = 0,
\quad
T_\u^- = 0, \quad T_\d^- = -T,
$$
for the end point of $I^\pm$ opposite to $t_{\u\d}^\pm$,
and define the nudging operators $N_{\u\d}^\pm$ by
$$
N_{\u\d}^\pm (U_0; \Lambda V_0) = (u(T_{\u\d}^\pm), \p_t u(T_{\u\d}^\pm)),
$$
where $u$ is the solution of (\ref{eq_wave_nudge})
with the initial conditions $U_0 = (u_0, u_1)$ at $t_{\u\d}^\pm$.

The iterative step of our time reversal method is given by the average of the two nudging cycles on $I^+$ and $I^-$ respectively, that is,
\begin{align*}
N(U_0; \Lambda V_0) 
&= \left(N^+(U_0; \Lambda V_0) + N^-(U_0; \Lambda V_0)\right)/2, 
\\
N^+(U_0; \Lambda V_0) 
&= N_\d^+ (N_\u^+(U_0; \Lambda V_0); \Lambda V_0),
\\
N^-(U_0; \Lambda V_0) 
&= N_\u^- (N_\d^-(U_0; \Lambda V_0); \Lambda V_0).
\end{align*}
We emphasize that the stabilizing boundary condition in (\ref{eq_wave_nudge}) is used only in the computational procedure, whereas the measurement data is generated by (\ref{eq_wave_meas}) with homogeneous Neumann boundary condition.

Before stating our results precisely, let us consider informally the iteration $N^j$, $j=1,2,\dots$, based on the nudging map $N(U_0) = N(U_0; \Lambda V_0)$.
The nudging map $N$ is a contraction (under suitable geometric conditions), since the
difference $w$ of two solutions of (\ref{eq_wave_nudge}), with different initial conditions, satisfies the stabilizing boundary condition $\p_\nu w + \chi \p_t w = 0$,
and this causes the energy of $w$ to decrease.
Thus $N$ admits a unique fixed point and the iteration $N^j U_0$ converges to the fixed point for any $U_0$. (Below we take $U_0 = 0$ for convenience.)
Moreover, the fixed point is $V_0$. Indeed, the solution $u$ of (\ref{eq_wave_nudge}) with the initial condition $(u_0, u_1) = V_0$ at the time $t=0$
coincides with the solution $v$ of (\ref{eq_wave_meas}) since 
$$
\p_\nu v + \chi \p_t v = \chi \p_t v = \chi \p_t \Lambda V_0.
$$
Hence $N^+_\u (V_0; \Lambda V_0) = (v(T), \p_t v(T))$, and similarly 
$N^+ (V_0; \Lambda V_0) = (v(0), \p_t v(0)) = V_0$.
The other cycle is analogous, and thus $N(V_0)=V_0$. 

Let us now formulate our main results. We define the domain of influence
$$
M(\Gamma, T) = \{x \in M;\ d(x, \Gamma) \le T\},
$$
where $d(\cdot, \cdot)$ is the distance function on $(M,g)$,
and use the shorthand notation
\begin{align*}
\H = \H(M(\Gamma, T)).
\end{align*}
Furthermore, we denote by $P_K$, $K \subset M(\Gamma,T)$, the orthogonal projection $\H \to \H(K)$ with respect to the energy inner product, that is, the inner product
associated to the norm (\ref{energy_norm}).

Due to the finite speed of propagation,  
it is necessary that $K \subset M(\Gamma, T)$
in order to fully recover $V_0 \in \H(K)$
given the measurement $\Lambda V_0$.
Our first result gives a time reversal iteration that converges to $V_0$ under this condition.

\begin{theorem}
\label{th_unstable}
Let $K \subset M^\inter$ be compact and suppose that $K \subset M(\Gamma,T)$.
Then for all $V_0 \in \H(K)$ it holds that $\lim_{j \to \infty} (P_K N)^j (0) = V_0$
where $N = N(\cdot; \Lambda V_0)$.
\end{theorem}

We say that a compact set $K \subset M^\inter$ satisfies 
the visibility condition if 
\begin{itemize}
\item[(VC)] For all $x \in K$ and all $\xi \in T_x K$ satisfying $|\xi| = 1$
there is $t \in (-T,T)$ such that 
$$
\gamma(t; x, \xi) = y \in \Gamma
\quad \text{and} \quad
\dot \gamma(t; x, \xi) \notin T_y \Gamma.
$$
Here $\gamma = \gamma(\cdot; x,\xi)$ is the geodesic of $(M,g)$ satisfying $\gamma(0) = x$
and $\dot \gamma(0) = \xi$,
and we assume that $\gamma$ does not intersect the boundary $\p M$ between times $0$ and $t$.
\end{itemize}

Let us point out that a slightly more general version of (VC) is necessary for continuity of the inverse 
$$
\Lambda^{-1} : H^1((-T,T) \times \Gamma) \to \H(K).
$$
Indeed, if geodesics $\gamma$ are extended by reflection or suitable gliding when they intersect $\p M \setminus \Gamma$,
see \cite{Bardos1992} for the precise definition,
and if there are $x \in K$ and $\xi \in T_x K$
such that $|\xi| = 1$ and that the corresponding extended $\gamma$ does not intersect $\Gamma$,
then $\Lambda^{-1}$ is not continuous \cite[Th. 3.2]{Bardos1992}.

Our second result says that the time reversal iteration converges to $V_0$
at an exponential rate, assuming the visibility condition.

\begin{theorem}
\label{th_stable}
Let a compact set $K \subset M^\inter$ 
satisfy the visibility condition (VC). 
Then there is a linear map $R : \H(K) \to \H(K)$
such that $\norm{R}_{\H(K) \to \H(K)} < 1$ and that
for all $V_0 \in \H(K)$
$$
(P_K N)^j (0) = V_0 - R^j V_0, \quad j=1,2,\dots,
$$
where $N = N(\cdot; \Lambda V_0)$.
\end{theorem}

\begin{remark}
Under the assumptions of Theorem \ref{th_stable},
an alternative reconstruction method is obtained by the Neumann series inversion
$$
V_0 = (1-R)^{-1} P_K N(0) = \sum_{j=0}^\infty R^j P_K N(0).
$$
Here $R = P_K S$, where $S$ is defined in Lemma \ref{lem_nudge} below.
\end{remark}

\section{Back and forth nudging identity}

We define $S^\pm_{\uparrow\downarrow}$ to be the solution operator at time $T_{\u\d}^\pm$ of the problem
\begin{align}
\label{eq_wave_S}
\begin{cases}
\p_t^2 u - \Delta_{g,\mu} u = 0, &\text{in $I^\pm \times M$},
\\
\p_\nu u  + \chi\p_t u = 0,
 &\text{in $I^\pm  \times \p M$},
 \\
 u|_{t = t_{\u\d}^\pm} = u_0,\ \p_t u|_{t = t_{\u\d}^\pm} = u_1, &\text{in $M$},
\end{cases}
\end{align}
where $\chi = \chi_0$ in the two cases $t_\u^\pm$
and $\chi = -\chi_0$ in the two cases $t_\d^\pm$.
That is,
$$
S^\pm_{\uparrow\downarrow} U_0 = (u(T_{\u\d}^\pm), \p_t u(T_{\u\d}^\pm)),
$$
where $u$ is the solution of (\ref{eq_wave_S})
with the initial conditions $U_0 = (u_0, u_1)$ at $t_{\u\d}^\pm$.

If $u_0 \in H^1(M)$ and $u_1 \in L^2(M)$ are compactly supported, then the equation (\ref{eq_wave_S}) has a unique solution $u$ in the energy space 
$C(I^\pm; H^1(M)) \cap C^1(I^\pm; L^2(M))$.
The sign of $\chi$ is important here, and in fact, the opposite choices of sign do not give well-posed problems.
For the convenience of the reader, we have included a study of (\ref{eq_wave_S}) in the appendix below.

\begin{lemma}
\label{lem_key}
Let $V_0 \in \H(K)$, $W_0 \in \H$ and define 
$$
V(t) = (v(t), \p_t v(t)), \quad t \in (-T,T),
$$
where $v$ is the solution of (\ref{eq_wave_meas}) with the initial conditions $V_0 = (v_0, v_1)$. Then 
\begin{align*}
V(T_{\u\d}^\pm) + S_{\u\d}^\pm (W_0 - V(t_{\u\d}^\pm)) &=
N^\pm_{\u\d}(W_0; \Lambda V_0).
\end{align*}
\end{lemma}
\begin{proof}
Let $u$ be the solution of (\ref{eq_wave_S}) with the initial conditions 
$$
(u|_{t=t_{\u\d}^\pm}, \p_t u|_{t=t_{\u\d}^\pm}) = W_0 - V(t_{\u\d}^\pm),
$$
and write $W_0 = (w_0, w_1)$. Then the sum $w = u + v$ satisfies 
\begin{align*}
\begin{cases}
\p_t^2 w - \Delta_{g,\mu} w = 0, &\text{in $I^\pm \times M$},
\\
\p_\nu w  + \chi\p_t w = \chi\p_t v,
 &\text{in $I^\pm \times \p M$},
 \\
 w|_{t = t_{\u\d}^\pm} = w_0,\ \p_t w|_{t = t_{\u\d}^\pm} = w_1, &\text{in $M$}.
\end{cases}
\end{align*}
We write $W(t) = (w(t), \p_t w(t))$, $t \in (0,T)$,
and observe that 
$W(T_{\u\d}^\pm) = N_{\u\d}^\pm(W_0; \Lambda V_0)$ since the above equation for $w$ coincides with the equation (\ref{eq_wave_nudge}).
The claim follows from
$$
W(T_{\u\d}^\pm) = S_{\u\d}^\pm(W_0 - V(t_{\u\d}^\pm)) + V(T_{\u\d}^\pm).
$$
\end{proof}

\begin{lemma}
\label{lem_nudge}
Let $V_0 \in \H(K)$ and $U_0 \in \H$. Then the following nudging identity holds
\begin{align}
\label{nudging_id}
N (U_0; \Lambda V_0) - V_0 = S (U_0 - V_0),
\end{align}
where $S = (S^+ + S^-)/2$,
$S^+ = S_\d^+ S_\u^+$ and $S^- = S_\u^- S_\d^-$.
\end{lemma}
\begin{proof}
We omit writing the parameter $\Lambda V_0$ in the proof. 
We apply Lemma \ref{lem_key} twice
\begin{align*}
N_\d^+ (N_\u^+ (U_0))
&= N_\d^+ (V(T) + S_\u^+(U_0 - V_0)) 
\\&= 
V_0 + S_\d^+((V(T) + S_\u^+(U_0 - V_0)) - V(T))
\\\notag&=
V_0 + S_\d^+ S_\u^+(U_0 - V_0).
\end{align*}
An analogous computation for the composition $N^- = N_\u^- N_\d^-$ shows that 
$$
N^- (U_0) - V_0 = S^- (U_0 - V_0),
$$
and we obtain (\ref{nudging_id})
after taking the average.
\end{proof}

\section{Energy identity and unique continuation}

The following energy identity is well-known, however, we give a short proof for the convenience of the reader.

\begin{lemma}
\label{lem_energy_decay}
Let a solution $u \in C^\infty((0,T) \times M)$ to
\begin{align*}
\begin{cases}
\p_t^2 u - \Delta_{g,\mu} u = 0, &\text{in $(0,T) \times M$},
\\
\p_\nu u  + \chi\p_t u = 0,
 &\text{in $(0,T)  \times \p M$},
 \end{cases}
\end{align*}
satisfy $\supp(u) \subset \R \times \mathcal K$ where $\mathcal K \subset M$ is compact. Then
\begin{align}
\label{energy_estimate}
\norm{U(t_1)}_*^2 
=
\norm{U(t_0)}_*^2
- \int_{t_0}^{t_1} \int_{\p M} \chi |\p_t u|^2 \, \mu dy dt, \quad t_0 < t_1,
\end{align}
where $U(t) = (u(t), \p_t u(t))$ and $dy$ is the Riemannian volume measure on $(\p M, g)$.
\end{lemma}
\begin{proof}
We first integrate by parts
\begin{align*}
\p_t \norm{U}_*^2 
&= 
\int_M \p_t^2 u\, \p_t u 
    + \pair{\mu \nabla u, \nabla \p_t u}\, dx 
\\
&= \int_M (\p_t^2 u - \Delta_{g,\mu} u)\, \p_t u\, \mu dx  
+ \int_{\p M} \p_\nu u\, \p_t u \, \mu dy
\\
&=
- \int_{\p M} \chi |\p_t u|^2 \, \mu dy,
\end{align*}
and then integrate in time
\begin{align*}
\norm{U(t_1)}_*^2 - \norm{U(t_0)}_*^2 = - \int_{t_0}^{t_1} \int_{\p M} \chi |\p_t u|^2 \, \mu dy dt.
\end{align*}
\end{proof}

\begin{lemma}
\label{lem_unique_cont}
Let a function $u$ in the energy space
\begin{align}
\label{energy_space_uc}
C((-T,T); H^1(M)) \cap C^1((-T,T); L^2(M))
\end{align}
satisfy the wave equation
\begin{align}
\label{eqw_uniq_cont}
\p_t^2 u - \Delta_{g,\mu} u = 0, \quad \text{in $(-T,T) \times M$},
\end{align}
and write $U_0 = (u|_{t=0}, \p_t u|_{t=0})$.
Let $K \subset M^\inter \cap M(\Gamma,T)$ be compact.
Suppose, furthermore, that $U_0 \in \H(K)$,
$\p_\nu u = 0$ on $(-T,T) \times \p M$ and that $\p_t u = 0$ on $(-T,T) \times \Gamma$.
Then $U_0 = 0$.
\end{lemma}
\begin{proof}
We write $U_0 = (u_0, u_1)$.
Suppose for the moment that $u \in C^\infty((-T,T) \times M)$.
The function $w = \p_t u$ satisfies (\ref{eqw_uniq_cont})
and both $w$ and $\p_\nu w$ vanish on $(-T,T) \times \Gamma$.
The semi-global version \cite[Th. 3.16]{Katchalov2001} of the time sharp unique continuation result \cite{Tataru1995}, implies that $w$ vanishes in the double cone
of points $(t,x)$ satisfying $d(x,\Gamma) \le T - |t|$.
In particular, $w|_{t=0} = u_1 = 0$ in $M(\Gamma, T)$.
The assumption $U_0 \in \H(K)$ implies that $u_1 = 0$ also outside $M(\Gamma, T)$. 
Furthermore, as $w$ vanishes in the double cone, also $\p_t w = \Delta_{g,\mu} u$ vanishes there.
Thus $\Delta_{g,\mu} u_0 = 0$ in $M(\Gamma, T)$. The assumption $U_0 \in \H(K)$ together with $K\subset M^\inter$ implies that $u_0$ vanishes near $\Gamma$, and therefore $u_0 = 0$ in $M(\Gamma, T)$ by elliptic unique continuation. Finally, $u_0 = 0$ outside $M(\Gamma, T)$ by the assumption $U_0 \in \H(K)$.

Let us now consider the general case where $u$ is in  (\ref{energy_space_uc}). 
Let $\psi \in C_0^\infty(-\epsilon/2,\epsilon/2)$, where $\epsilon >0$ is small.
We extend $u$ by zero to a function on $\R \times M$ and denote the extension still by $u$. The convolution in time $w = \psi * \p_t u$ satisfies (\ref{eqw_uniq_cont}) 
on 
$(-T+\epsilon, T - \epsilon) \times M$,
and the Cauchy data $(w,\p_\nu w)$ vanishes on 
$(-T+\epsilon, T - \epsilon) \times \Gamma$.
Note that $w \in C^\infty(\R; L^2(M))$. Let $t \in \R$
satisfy $|t| < T - \epsilon$. Then
\begin{align}
\label{Delta_w}
\Delta_{g,\mu} w(t) = \p_t^2w(t) \in L^2(M).
\end{align}
As the spatial support of $u$ is contained in the compact set $\mathcal K$, there is a compact manifold with smooth boundary $M_0 \subset M$ such that $M(\Gamma, T) \subset M_0$
and $\p_\nu w(t) = 0$ on $\p M_0$. As $w(t)$ satisfies (\ref{Delta_w}) with vanishing  Neumann boundary conditions on $M_0$, it holds that $w(t) \in H^2(M_0)$, see e.g. \cite[Prop. 7.6]{Taylor1996}.
Hence we may apply \cite[Th. 3.16]{Katchalov2001} on $w$.
As above, we see that 
both $\psi * u_1$ 
and $\Delta_{g,\mu} (\psi * u_0)$ vanish on $M(\Gamma, T-\epsilon)$. Letting $\phi \to \delta$ and $\epsilon \to 0$,
we get 
$$
u_1 = \Delta_{g,\mu} u_0 = 0 \quad \text{on $M(\Gamma, T)^\inter$}.
$$
Using the assumption $U_0 \in \H(K)$ as above, we see that $U_0 = 0$.
\end{proof}

\begin{lemma}
\label{lem_energy_S}
Let $K \subset M^\inter \cap M(\Gamma,T)$ be compact.
Let $U_0 \in \H(K)$, and consider the function
\begin{equation*}
u(t,x) = \begin{cases}
u^+(t, x), & t \ge 0,
\\
u^-(t, x), & t \le 0,
\end{cases}
\end{equation*}
where $u^\pm$ are the solutions of (\ref{eq_wave_S})
with the initial conditions given by $U_0$ at $t_\u^+$
and $t_\d^-$ respectively. Then
$$
\norm{S U_0}_*^2 
\le 
\norm{U_0}_*^2 
- \frac 1 2 
\int_{-T}^{T} \int_{\p M} \chi_0 |\p_t u|^2 \, \mu dy dt.
$$
Moreover, the continuous map
$$
E : \H(K) \to L^2((-T,T) \times \p M), \quad
E U_0 = \sqrt{\chi_0} \p_t u,
$$
is injective.
\end{lemma}
\begin{proof}
Suppose for the moment that $U_0 \in C_0^\infty(K)^2$.
Then both $u^\pm$ are smooth, see Theorem \ref{th_solvability} in the appendix below.
As $\chi_0$ is non-negative, the energy identity (\ref{energy_estimate}) implies that $$\norm{S_{\u\d}^\pm U_0}_* \le \norm{U_0}_*.$$ 
Note that for the backward propagators $S_\d^\pm$
we use the fact that $\chi=-\chi_0$.
It holds that
\begin{align*}
\norm{S U_0}_*^2
&\le \frac 1 4 
\left(\norm{S_\d^+ S_\u^+ U_0}_* + \norm{S_\u^- S_\d^- U_0}_* \right)^2
\le \frac 1 2 
\left(\norm{S_\u^+ U_0}_*^2 + \norm{S_\d^- U_0}_*^2 \right)
\\&=
 \frac 1 2 
\left( \norm{U_0}_*^2 - \int_0^T \int_{\p M} \chi_0 |\p_t u|^2 \, \mu dy dt + \norm{U_0}_*^2 - \int_{-T}^0 \int_{\p M} \chi_0 |\p_t u|^2 \, \mu dy dt \right)
\\&= 
\norm{U_0}_*^2 - \frac 1 2 \int_{-T}^T \int_{\p M} \chi_0 |\p_t u|^2 \, \mu dy dt.
\end{align*}
The general case $U_0 \in \H(K)$ follows since $C_0^\infty(K)^2$ is dense in $\H(K)$.

Suppose now that $E U_0 = 0$. 
Then $\p_\nu u = 0$ on $(-T, T) \times \p M$
and, using the assumption $\supp(\chi_0) = \overline \Gamma$, we see that $\p_t u = 0$ on $(-T,T) \times \Gamma$. 
Note that by Theorem \ref{th_solvability} (cf. Appendix) $u^+$ is in the energy space (\ref{solvability_esp}) and $u^-$ is in the corresponding energy space on $[-T, 0]$. As $K\subset M^{int}$, we see that $u$ belongs to the space (\ref{energy_space_uc}).
Lemma \ref{lem_unique_cont} implies that $U_0 = 0$.
\end{proof}

\section{The unstable case}

\begin{lemma}
\label{lem_S}
The operator $P_K S : \H(K) \to \H(K)$ is self-adjoint and positive.
\end{lemma}
\begin{proof}
%
Let us show that $P_K S^+$ is self-adjoint and positive. The case $P_K S^-$ is similar and we omit it.
We write $U_{\u\d}(t) = (u_{\u\d}(t), \p_t u_{\u\d}(t))$
where $u_\u$ and $u_\d$ are smooth solutions of 
\begin{align}
\label{weq_lem_S}
\begin{cases}
\p_t^2 u - \Delta_{g,\mu} u = 0, &\text{in $(0,T) \times M$},
\\
\p_\nu u  + \chi\p_t u = 0,
 &\text{in $(0,T)  \times \p M$},
\end{cases}
\end{align}
with $\chi = \chi_0$ and $\chi = -\chi_0$, respectively.
Then, writing $\pair{\cdot, \cdot}_M$ and $\pair{\cdot, \cdot}_{\p M}$ for the $L^2$ inner products on $M$ and $\p M$ with respect to the measures used in Lemma \ref{lem_energy_decay},
\begin{align*}
\p_t \pair{U_\u, U_\d}_*
&= 
\pair{\p_t^2 u_\u, \p_t u_\d}_M
+ \pair{\nabla u_\u, \nabla \p_t u_\d}_M
+ \pair{\p_t u_\u, \p_t^2 u_\d}_M
+ \pair{\nabla \p_t u_\u, \nabla u_\d}_M
\\&= 
\pair{\chi_0 \p_t u_\u, \p_t u_\d}_{\p M}
- \pair{\p_t u_\u, \chi_0 \p_t u_\d}_{\p M} = 0.
\end{align*}
Hence $\pair{U_\u(0), U_\d(0)}_* 
= 
\pair{U_\u(T), U_\d(T)}_*$.

Let $U_0 \in C_0^\infty(K)^2$, and let $U_\u(t) = (u_\u, \p_t u_\u(t))$ where $u_\u$ satisfies (\ref{weq_lem_S}) with $\chi = \chi_0$ and also the initial condition $U_\u(0) = U_0$.
Then $U_\u$ is smooth, see 
Theorem \ref{th_solvability} in the appendix below.
Set $U^r_\d(t) = (u(t), \p_t u(t))$ where $u(t) = u_\u(2T-t)$.
Then $u$ satisfies (\ref{weq_lem_S})
with $\chi = -\chi_0$, and also 
$U^r_\d(T) = (u_\u(T), -\p_t u_\u(T))$. 
Let $\phi_j \in C_0^\infty(M)$ converge to $2\p_t u_\u(T)$
in $L^2(M)$, and write $U_\d^j(t) = (u_\d^j(t), \p_t u_\d^j(t))$ where $u^j$ is the solution of (\ref{weq_lem_S}) with $\chi = -\chi_0$ satisfying 
$U_\d^j(T) = (0, \phi_j)$. 
Denote by $U_\d$ the solution of (\ref{weq_lem_S}) with $\chi = -\chi_0$ satisfying 
$U_\d(T) = U_\u(T) = S^+_\u U_0$.

Let $W_0 \in C_0^\infty(K)^2$. 
As $U^r_\d$ and $U_\d^j$, $j \in \N$, are smooth, we have
\begin{align*}
\pair{W_0, S^+_\d S^+_\u U_0}_*
&= \pair{W_0, U_\d(0)}_* 
= \lim_{j \to \infty}  \pair{W_0, U^r_\d(0) + U_\d^j(0)}_*
\\&=
\lim_{j \to \infty}  \pair{S^+_\u W_0, U^r_\d(T) + U_\d^j(T)}_*
= \pair{S^+_\u W_0, S^+_\u U_0}_*.
\end{align*}
The above computation, with the roles of $W_0$ and $U_0$
interchanged, implies that 
$$
\pair{S^+_\d S^+_\u W_0, U_0}_* = \pair{S^+_\u W_0, S^+_\u U_0}_* = \pair{W_0, S^+_\d S^+_\u U_0}_*.
$$
The same holds for all $U_0, W_0 \in \H(K)$ by density, which again shows that $P_K S^+$ is self-adjoint and positive.
\end{proof}
The following proof is an adaptation of the proof of \cite[Th. 1.1]{Haine2014}. The main difference is that
the result in \cite{Haine2014} is given in an abstract setting that does not allow for a full characterization of
the convergence.

\begin{proof}[Proof of Theorem \ref{th_unstable}]
By Lemma \ref{lem_nudge},
$$
P_K N(U_0) - V_0 = 
P_K (N(U_0) - V_0) = P_K S(U_0 - V_0), \quad U_0 \in \H.
$$
We iterate this starting from $U_0 = 0$. That is, for $j=1,2,\dots$,
\begin{align}
\label{N_iteration}
(P_K N)^j(0) - V_0 = P_K S((P_K N)^{j-1}(0) - V_0)
= -(P_K S)^j V_0.
\end{align}
Writing $R = P_K S$, it remains to show that the sequence $R^j V_0$, $j\in\mathbb N$, converges to zero.

By Lemma \ref{lem_S}, the operator $R$ is self-adjoint and positive, and whence also the powers $R^j$, $j \in \N$, are self-adjoint and positive.
We will use the fact that a non-increasing sequence of bounded positive operators in a Hilbert space converges pointwise to a bounded operator, see e.g. \cite[Lem. 12.3.2]{tucsnak2009observation}.

Lemma \ref{lem_energy_S} implies that $\norm{SV_0}_* \le \norm{V_0}_*$, $V_0 \in \H(K)$, and that the equality holds only if $V_0 = 0$. Hence
$$
\norm{R(R^{j/2}V_0)}_* = \norm{P_K S(R^{j/2}V_0)}_* 
\le \norm{R^{j/2}V_0}_*, \quad j \in \N,\ V_0 \in \mathcal H(K),
$$
which again implies that the sequence $R^{j}$, $j \in \N$, is non-increasing. Indeed,
\begin{align*}
\langle R^{j+1}V_0,V_0\rangle_* 
&= 
\langle R (R^{j/2}V_0),R^{j/2}V_0\rangle_* \le \|R(R^{j/2}V_0)\|_*\|R^{j/2}V_0\|_*
\\&\le
\|R^{j/2}V_0\|_*^2 
= \langle R^{j/2}V_0,R^{j/2}V_0 \rangle_* = \langle R^j V_0, V_0 \rangle_*, \quad V_0 \in \mathcal H(K).
\end{align*}
Therefore the sequence $R^j$, $j\in\mathbb N$, converges pointwise on $\mathcal H(K)$, say to an operator $R^\infty$,
and it remains to prove that $R^\infty = 0$.

Observe that $RR^\infty = R^\infty$, since for all $V_1, V_2\in\mathcal H(K)$
\begin{equation*}
\langle RR^\infty V_1, V_2\rangle_*=
\lim_{j\to\infty}\langle RR^jV_1, V_2\rangle_*=
\lim_{j\rightarrow\infty}\langle R^{j+1} V_1, V_2\rangle_*=\langle R^\infty V_1, V_2\rangle_*.
\end{equation*}
To get a contradiction, we assume that there is $V_0\in\mathcal H(K)$ such that $R^\infty V_0\ne0$.
Then Lemma \ref{lem_energy_S} implies 
$$
\|R^\infty V_0\|^2_*
=
\langle RR^\infty V_0, R^\infty V_0\rangle_*
\le
\norm{P_K S R^\infty V_0}_* \norm{R^\infty V_0}_*
< \norm{R^\infty V_0}_*^2,
$$
which is indeed a contradiction.
\end{proof}

\section{The stable case}

In this section we prove Theorem \ref{th_stable}
using techniques that are similar to those in \cite{Bardos1992}.
The main difference between Lemma \ref{lem_microlocal} below and \cite[Th. 3.3]{Bardos1992} is that we give a local result under the local geometric assumption (VC), whereas \cite{Bardos1992} gives a global result under a global geometric assumption.

We will use the notion of microlocal regularity at a point, see e.g. \cite[Def. 18.1.30]{Hormander1985}.
Recall that a distribution $u$ on $M$ is said to be microlocally $H^s$ at $(x,\xi) \in T^*M \setminus 0$
if $u = u_0 + u_1$ where $(x,\xi) \notin \WF(u_0)$
and $u_1 \in H^s_{loc}(M)$.
In this case we write $u \in H^s(x,\xi)$.
Here $\WF$ denotes the wave front set, see \cite[Def. 8.1.2]{Hormander1990}, and $H^s_{loc}(M)$, $s \in \R$, is the space of distributions $w$ such that $\phi w$ is in the Sobolev space $H^s(M)$ for all $\phi \in C_0^\infty(M)$.

If $u \in H^s(x,\xi)$ then 
by \cite[Th. 18.1.31]{Hormander1990} 
there is a pseudodifferential operator of order zero $A$ such that $A u \in H^s_{loc}(M)$ and $(x,\xi)$ is not characteristic for $A$.
Moreover, by \cite[Th. 18.1.24']{Hormander1990} 
there is a pseudodifferential operator of order zero $B$
such that 
\begin{align}
\label{I_BA}
(x,\xi) \notin \WF(1 -BA)
\end{align}
where $1$ is the identity map.
Note also that after multiplication by a function $\phi \in C_0^\infty(M)$
satisfying $\phi=1$ near $x$, we may assume that $A u \in H^s(M)$.

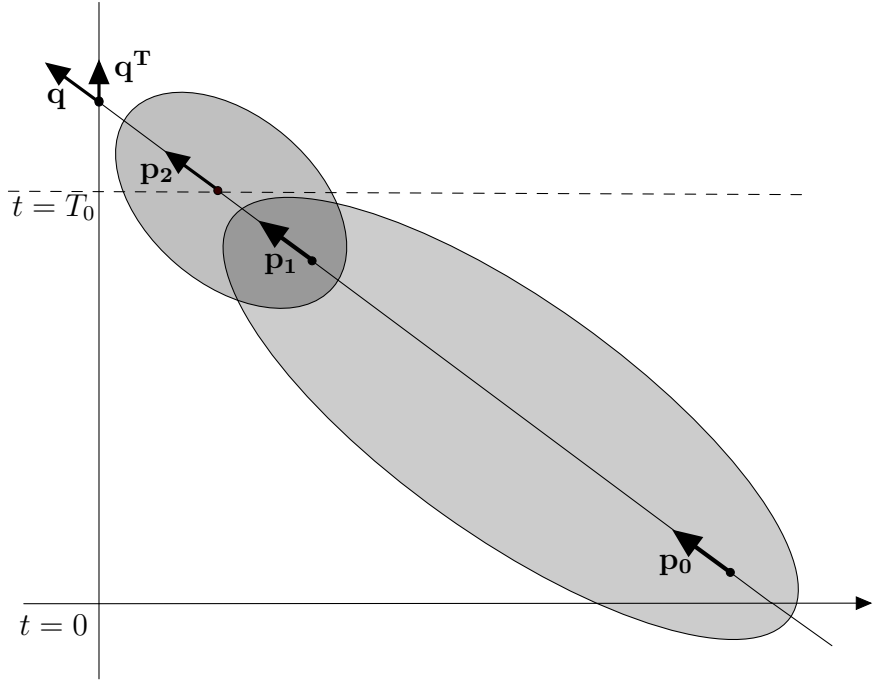
\begin{figure}
\definecolor{ttqqqq}{rgb}{0.2,0.,0.}
\definecolor{qqqqtt}{rgb}{0.,0.,0.2}
\begin{tikzpicture}[line cap=round,line join=round,>=triangle 45,x=1.0cm,y=1.0cm]
\clip(-3.54934,-4.280447239643131) rectangle (8.63064323290623,5.590903835887293);
\draw [->] (-2.,-3.) -- (8.279965554581548,-2.992526375496311);
\draw (-2.,3.6737965245191564)-- (6.943245570316854,-2.9934981825686138);
\draw (-2.8359224517467085,4.) node[anchor=north west] {$\mathbf{q}$};
\draw (-1.6,3.) node[anchor=north west] {$\mathbf{p_2}$};
\draw (5.3,-2.2) node[anchor=north west] {$\mathbf{p_0}$};
\draw [->,line width=1.2pt] (-0.410499975358124,2.499143149937045) -- (-1.1310284049942332,3.0274684553346303);
\draw (-1.9409077093925764,4.508239524377504) node[anchor=north west] {$\mathbf{q^T}$};
\draw (-3.,-3.)-- (-2.,-3.);
\draw (-3.2,-3) node[anchor=north west] {$t=0$};
\draw (-2.,-3.)-- (-2.,5.);
\draw (-2.,-3.)-- (-2.,-4.);
\draw [->,line width=1.6pt] (6.393245107783388,-2.5834663720907445) -- (5.619551077612657,-2.0155054864939372);
\draw [->,line width=1.2pt] (-2.,3.6737965245191564) -- (-2.72052842963611,4.20212182991674);
\draw [->,line width=1.6pt] (0.8381951676062861,1.5689763867012236) -- (0.1176667379701759,2.0973016920988066);
\draw (0.05,1.8) node[anchor=north west] {$\mathbf{p_1}$};
\draw (-3.3,2.595) node[anchor=north west] {$t=T_0$};
\draw (6.940691997840085,-2.984900254961839)-- (7.743927336241997,-3.5589886046401067);
\draw [->,line width=1.2pt] (-2.,3.689331869279828) -- (-2.,4.245083513501792);
\draw [dash pattern=on 4pt off 4pt] (-3.204618606025738,2.4877066925836893)-- (7.339642319482882,2.442645748457584);
\draw [rotate around={144.46771925053451:(3.4733034808294274,-0.5351343467803934)},fill=black,fill opacity=0.2] (3.4733034808294274,-0.5351343467803934) ellipse (4.563892421796626cm and 1.560946946813927cm);
\draw [rotate around={139.39594828546296:(-0.24306352823262412,2.3671309209470253)},fill=black,fill opacity=0.25] (-0.24306352823262412,2.3671309209470253) ellipse (1.7820139930089436cm and 1.1184494352127314cm);
\begin{scriptsize}
\draw [fill=qqqqtt] (-2.,3.6737965245191564) circle (1.5pt);
\draw [fill=ttqqqq] (-0.42044415803143,2.496219075493177) circle (1.5pt);
\draw [fill=black] (6.393245107783388,-2.5834663720907445) circle (1.5pt);
\draw [fill=black] (0.8328820849188356,1.5618496276037517) circle (1.5pt);
\draw [fill=black] (-2.,3.689331869279828) circle (1.5pt);
\end{scriptsize}
\end{tikzpicture}
\label{cutoffs}
\caption{The localization and microlocalization. Here the large grey ellipse indicates the set on which $\psi = 1$, and the small grey ellipse indicates the base space projection of the set on which (\ref{I_BA}) holds.}
\end{figure}

\begin{lemma}
\label{lem_microlocal}
Suppose that the condition (VC) holds for a compact set $K \subset M^\inter$. 
Let $u \in L^2((-T,T) \times M)$
satisfy the wave equation
$$\p_t^2 u - \Delta_{g,\mu} u = 0~~ \mathrm{in}~(-T,T) \times M,$$
and suppose that
$\partial_tu|_{x \in \Gamma}, \p_\nu u|_{x \in \Gamma} \in L_{loc}^2((-T,T) \times \Gamma)$.
Then $u |_{\{t=0\} \times K} \in H^1(K)$ and 
$\p_t u |_{\{t=0\} \times K} \in L^2(K)$.
\end{lemma}
\begin{proof}
By compactness of $K$ there is a neighbourhood $\tilde K \subset M^\inter$ of $K$ and $\epsilon > 0$
such that (VC) holds with $K$ replaced by $\tilde K$ and $T$ replaced by $T-\epsilon$.
Finite speed of propagation and the associated energy estimate imply that it is enough to show that $u\in H^1((-\epsilon, \epsilon) \times \tilde K)$. By using a microlocal partition of unity, it is enough to show that $u\in H^1(p_0)$ at each point $p_0 \in T^*((-\epsilon, \epsilon) \times \tilde K)$, see \cite[Th. 18.1.31]{Hormander1985}. 
Furthermore, it is enough to consider characteristic points for the wave operator, that is, points in the set 
$
\{ (t,x,\tau,\xi) \in T^*(\R \times M^\inter) \setminus 0;\ 
\tau^2 = |\xi|^2 \},
$
since $\WF(u)$ is contained in this set.

We assume now that $p_0\in T^*((-\epsilon, \epsilon) \times \tilde K) $ is a characteristic point and let $\gamma$ be the bicharacteristic passing through $p_0$.
The visibility condition (VC) says that 
$\gamma$ intersects the boundary $\p T^* (\R \times M)$
non-tangentially at a point $q$ that lies over the set $(-T,T) \times \Gamma$.
We write $q = (t_0, y_0, 0; \tau_0, \eta_0, \sigma_0)$ in such local coordinates $(y, s)$ of $M$ that $s = 0$ describes the boundary $\p M$. Then (VC) means that $y_0 \in \Gamma$ and $\sigma_0 \ne 0$.
We may assume without loss of generality that $t_0 \in (0,T)$, the case $t_0 \in (-T,0)$ being analogous.

We denote by $q^T=(t_0, y_0; \tau_0, \eta_0)$ the projection
of $q$ on $T^* (\R \times \p M)$,
and consider the time derivative $\partial_t$ as an operator on the boundary $\R \times \p M$. As $\tau_0 \ne 0$, we see that $q^T$ is not characteristic for $\p_t$.
Together with the assumption $\partial_tu\in L_{loc}^2((-T,T) \times \Gamma)$
this implies that $u\in H^1(q^T)$,
see \cite[Th. 18.1.31]{Hormander1985}.

\def\U{\mathcal U}
Since $\sigma_0 \ne 0$, we have that $\tau_0^2 \ne |\eta_0|_g^2$, and therefore $q^T$ is not characteristic for the wave operator. By \cite[Th. 2.1]{Bardos1992} there is 
a tangential pseudodifferential operator 
$A = A(t,y,D_t,D_y)$
of order zero and a neighbourhood $\U \subset \R \times M^\inter$ of $(t_0, y_0, 0)$ such that 
$q^T$ is not characteristic for 
$A$ 
and $Au\in H^1(\U)$.
The set of characteristic points is closed and this, together with the fact that $A$ is tangential, implies that
there is $p_1 \in \gamma \cap T^*(\R \times M^\inter)$ such that $p_1$ is not characteristic for $A$. After multiplication by a cut-off function, we may assume that $Au \in H^1(\R \times M)$.

For two distinct points $p, \tilde p \in \gamma$,
we denote the segment of $\gamma$ between $p$ and $\tilde p$
by $\sigma(p, \tilde p)$.
As discussed above, there is a pseudodifferential operator of order zero $B$ such that 
(\ref{I_BA}) holds with $(x,\xi)$ replaced by $p_1$. 
As the wave front set is closed, there is a point $p_2 \in \sigma(p_1, q)$,
distinct from $p_1$ and $q$, such that 
(\ref{I_BA}) holds with $(x,\xi)$ replaced by 
any point in the segment $\sigma(p_1, p_2)$.
We denote by $T_0$ the time coordinate of $p_2$,
and choose $\psi \in C_0^\infty((-T,T_0) \times M^\inter)$
satisfying $\psi = 1$ in a neighbourhood of $\sigma(p_0, p_1)$, see Figure \ref{cutoffs}.

We define $\Phi = [\Box, \psi] u$
where $[\Box, \psi]$ is the commutator of 
 $\Box = \p_t^2 - \Delta_{g,\mu}$ and $\psi$. Then $\psi u$ satisfies the equation
\begin{align}
\label{Box}
\begin{cases}
\Box u = F, &\mathrm{in}~(-T, T_0)\times M,\\
u=0, &\mathrm{in}~(-T, T_0)\times\partial M,\\
u|_{t=T_0} = \p_t u|_{t=T_0} =0, &\mathrm{in }~M,\\
\end{cases}
\end{align}
with $F = \Phi$.
We consider also the solutions $u_0$ and $u_1$ of 
(\ref{Box}) with $F=(1-B A)\Phi$ and $F=B A\Phi$,
respectively.
Clearly, $\psi u = u_0 + u_1$ and we will establish $u \in H^1(p_0)$ by showing that that $p_0 \notin \WF(u_0)$ and $u_1 \in H^1((-T,T_0) \times M)$.

Let us consider first $u_0$. We have $\sigma(p_0, p_2) \subset T^*(\R \times M^\inter)$ and $u_0 = 0$ for $T_0 > 0$. 
In view of propagation singularities in the interior 
\cite[Th. 23.2.9]{Hormander1985}, it is enough to show that 
$$
\WF((1-BA) \Phi) \cap \sigma(p_0, p_2) = \emptyset.
$$
Now (\ref{I_BA}) implies that $\WF((1-BA) \Phi) \cap \sigma(p_1, p_2) = \emptyset$. Moreover, $\Phi = 0$ near 
$\sigma(p_0, p_1)$ since $\psi = 1$ there, which again implies that 
$\WF((1-BA) \Phi) \cap \sigma(p_0, p_1) = \emptyset$.

We turn to $u_1$. We have
$$
A \Phi = [\Box, \psi] A u + [A, [\Box, \psi]] u \in L^2((-T,T_0) \times M),
$$
since 
$A u \in H^1((-T,T_0) \times M)$,
$u \in L^2((-T,T_0) \times M)$,
and the commutators 
$[\Box, \psi]$ and $[A, [\Box, \psi]]$ 
are of order one and zero, respectively.
As $B$ is of order zero, we have 
$$
BA \Phi \in L^2((-T,T_0) \times M),
$$
and the energy estimate for (\ref{Box}) implies that $u_1 \in H^1((-T,T_0) \times M)$.
\end{proof}

Theorem \ref{th_stable} follows from (\ref{N_iteration})
and the below lemma.

\begin{lemma}
\label{lem_norm_stable}
Suppose that the condition (VC) holds for a compact set $K \subset M^\inter$. 
Then $\norm{S}_{\H(K) \to \H} < 1$.
\end{lemma}
\begin{proof}
We define two spaces as follows
\begin{align*}
X&=\{u\in L^2((-T,T)\times M);\ \Box u=0,\ \sqrt{\chi_0} \partial_tu, \partial_\nu u\in L^2((-T,T)\times\Gamma)\},
\\
Y&=\{u\in L^2((-T,T)\times M);\ \Box u=0,\ 
u|_{\{t=0\} \times K} \in H^1(K),\ \partial_tu|_{\{t=0\} \times K}\in L^2(K)\}.
\end{align*}
They are Banach spaces with respect to norms,
\begin{align*}
\|u\|_X &=\|u\|_{L^2((-T,T) \times M)}+\|\sqrt{\chi_0} \partial_tu\|_{L^2((-T,T) \times\Gamma)}+\|\partial_\nu u\|_{L^2((-T,T)\times\Gamma)},
\\
\|u\|_Y &=\|u\|_{L^2((-T,T) \times M)}
+\|u|_{\{t=0\} \times K}\|_{H^1(K)}
+\|\partial_tu|_{\{t=0\} \times K}\|_{L^2(K)},
\end{align*}
and Lemma \ref{lem_microlocal} implies that $X\subset Y$.
The identity map is closed from $X$ to $Y$ since
both the norms are stronger than 
that of $L^2((-T,T))\times M)$.
The closed graph theorem implies that there is $C>0$ such that
$
\|u\|_Y\le C\|u\|_X
$
for all $u\in X$.

Let $U_0 \in \H(K)$ and consider the function $u$ defined in Lemma \ref{lem_energy_S}. Then $\sqrt{\chi_0} \partial_tu$ is in $L^2((-T,T) \in \p M)$. The boundary condition in (\ref{eq_wave_S}) implies that $u \in X$ and, together with $U_0 \in \H(K)$, it implies that the estimate
$\|u\|_Y\le C\|u\|_X$ reduces to
\begin{align}
\label{before_comp_uniq}
\norm{U_0}_* \le C \left( \|\sqrt{\chi_0} \partial_tu\|_{L^2((-T,T) \times\Gamma)} + \norm{u}_{L^2(-T,T) \times M)} \right).
\end{align}
As $\supp(U_0) \subset K$, we can choose a compact set $M_0 \subset M$ such that $u$ is supported in $[-T,T] \times M_0$.
The embedding $H^1((-T,T) \times M_0) \subset L^2((-T,T) \times M_0)$ is compact, and whence the map 
$U_0 \mapsto u$
is compact from $\H(K)$ to $L^2((-T,T) \times M_0)$.
By Lemma \ref{lem_energy_S},
the map $U_0 \mapsto \sqrt{\chi_0} \p_t u|_{(-T,T) \times \Gamma}$ is injective.
Thus the estimate (\ref{before_comp_uniq}) improves to 
\begin{align}
\label{after_comp_uniq}
\norm{U_0}_* \le C \norm{\sqrt{\chi_0} \p_t u}_{L^2((-T,T) \times \Gamma)},
\end{align}
by using the compactness-uniqueness argument. For the convenience of the reader, we have included 
the proof of this argument, see Lemma \ref{lem_comp_uniq} below. The claimed norm estimate follows now from Lemma \ref{lem_energy_S}.
\end{proof}

\begin{lemma}[Compactness-uniqueness]
\label{lem_comp_uniq}
Let $X$, $Y$ and $Z$ 
be Banach spaces and let $A : X \to Y$ and
$K : X \to Z$ be continuous linear maps. 
Suppose that $A$ is injective, $K$ is compact and that there is $C > 0$ such that 
$$
\norm{x}_X \le C \norm{Ax}_Y + C \norm{Kx}_Z, \quad x \in X.
$$
Then there is $C > 0$ such that 
$\norm{x}_X \le C \norm{Ax}_Y$ for $x \in X.$
\end{lemma}
\begin{proof}
To get a contradiction suppose that there is a sequence $x_j \in X$, $j \in \N$, such that $\norm{x_j}_X = 1$ for all $j$ 
and that $\norm{Ax_j}_Y \to 0$ as $j \to \infty$.
By compactness of $K$, there is a subsequence, still denoted by $x_j$, $j \in \N$, such that $K x_j$, $j \in \N$,
is a Cauchy sequence.
Hence
\begin{align*}
C^{-1} \norm{x_j - x_k}_X 
&\le \norm{A(x_j-x_k)}_Y + \norm{K(x_j - x_k)}_Z 
\\&
\le \norm{A x_j}_Y + \norm{A x_k}_Y + \norm{K x_j - K x_k}_Z \to 0, \quad j,k \to \infty.
\end{align*}
Thus $x_j$, $j \in \N$, is a Cauchy sequence, and therefore it converges, say to $x$.
We have 
$$
\norm{x}_X = \lim_{j \to \infty} \norm{x_j}_X = 1,
\quad \norm{Ax}_Y = \lim_{j \to \infty} \norm{Ax_j}_Y = 0,
$$
which is a contradiction to the injectivity of $A$.
\end{proof}

\appendix
\section*{Appendix. Wave equation with stabilizing boundary conditions}
\label{appendix}

In this appendix we consider the problem (\ref{eq_wave_S}).
For simplicity, we consider only the case $\mu = 1$ and write $\Delta = \Delta_{g,\mu}$. The modifications required in the case of a non-constant $\mu$ are straightforward.
The following theorem is well-known. For the convenience of the reader we give here a short proof with a reduction to the results in \cite{ikawa1970mixed}.
\begin{theorem}
\label{th_solvability}
Let $\chi \in C^\infty(\p M)$ be non-negative.
Then the equation 
\begin{align}
\label{weq_solvability}
\begin{cases}
\p_t^2 u - \Delta u = 0, &\text{in $(0,T) \times M$},
\\
\p_\nu u  + \chi\p_t u = 0,
 &\text{in $(0,T)  \times \p M$},
\\
 u|_{t = 0} = u_0,\ \p_t u|_{t = 0} = u_1, &\text{in $M$},
 \end{cases}
\end{align}
has a unique solution $u$ in the space
\begin{align}
\label{solvability_esp}
\mathcal E=C([0,T]; H^1(M)) \cap C^1([0,T]; L^2(M))
\end{align}
for any $u_0 \in H^1(M)$ and $u_1 \in L^2(M)$
with compact $\supp(u_j) \subset M$, $j=0,1$.
Moreover, if $u_0, u_1 \in C_0^\infty(K)$ and $K \subset M^\inter$ is compact, then the solution $u$ is smooth.
\end{theorem}

\begin{proof}
If $u_0, u_1 \in C_0^\infty(K)$ then the claim immediately follows from \cite{ikawa1970mixed}. Choose $u_0\in H^2(M)$ such that $\partial_\nu u_0=0$ on $\partial M$, and $u_1\in H_0^1(M)$. Then \cite{ikawa1970mixed} implies that (\ref{weq_solvability}) has unique solution $u\in H^2((0, T)\times M)$.

We define the energy as
$$
E(t) = \int_M |u(t)|^2 + |\p_t u(t)|^2 + |\nabla u(t)|^2\, \mu dx,
$$
and, analogously to the proof of Lemma \ref{lem_energy_decay}, we have
\begin{align*}
\p_t E  
\le C E + \norm{\p_t^2 u - \Delta u}_{L^2(M)}^2.
\end{align*}
Gr\"onwall's inequality implies that 
\begin{align}
\label{solvability_energy}
E(t) \le C E(0) + C \norm{\p_t^2 u - \Delta u}_{L^2((0,T) \times M)}^2, \quad t \in (0,T).
\end{align}
Thus,
$$
\norm{u}_{\mathcal{E}} \le C \norm{u_0}_{H^1(M)} + C \norm{u_1}_{L^2(M)}.
$$
Lemma \ref{lem_density} below implies that the map $(u_0, u_1) \mapsto u$ has a unique continuous extension 
$$H^1(M) \times L^2(M) \to \mathcal{E},$$
which concludes the proof of the theorem.
\end{proof}

\begin{lemma}
\label{lem_density}
Let $u \in H^1(M)$. Then there is a sequence $u_j \in H^2(M)$, $j \in \N$, such that $u_j \to u$ in $H^1(M)$ and $\p_\nu u_j = 0$ on $\p M$ for all $j \in \N$.
\end{lemma}
\begin{proof}
Consider a right inverse of the trace map 
$$
P : H^{3/2}(\p M) \times H^{1/2}(\p M) \to H^2(M),
\quad \p_\nu^k P(h_0, h_1)|_{\p M} = h_k, \quad k=0,1.
$$
Let a sequence $\phi_j \in C^\infty(M)$, $j \in \N$,
converge to $u$ in $H^1(M)$.
Write $w_j = P(\phi_j, 0)$ and $\tilde u_j = \phi_j - w_j$.
Then $\tilde u_j \in H_0^1(M)$ and there is $\psi_j \in C_0^\infty(M)$ such that $\norm{\psi_j - \tilde u_j}_{H^1(M)} \le 1/j$.
Set $u_j = \psi_j + w_j$. Then $u_j \in H^2(M)$, 
$\p_\nu u_j = \p_\nu w_j = 0$ on $\p M$
and
$$
\norm{u_j - u}_{H^1(M)}
\le 
\norm{\psi_j - \tilde u_j}_{H^1(M)} 
+ \norm{\tilde u_j + w_j - u}_{H^1(M)}
\le 1/j + \norm{\phi_j - u}_{H^1(M)} \to 0.
$$
\end{proof}

\noindent{\bf Acknowledgements.}
The research was supported by EPSRC grant EP/L026473/1.
The authors thank C. Laurent for introducing them to \cite{Ito2011}, and G. Haine and B. Cox for useful discussions. Moreover, the authors thank P. Stefanov for bringing \cite{Nguyen2015} to their attention. The authors are grateful to the anonymous referees for their comments, which helped to improve the paper.

\bibliographystyle{abbrv}
\bibliography{main}

\ifdraft{
\listoftodos
}{}

\end{document}